\documentclass[11pt]{amsart}
\usepackage{graphicx}
\usepackage{amsmath}
\usepackage{color}

\newtheorem{thm}{Theorem}[section]

\newtheorem{cor}[thm]{Corollary}
\newtheorem{lem}[thm]{Lemma}
\newtheorem{prop}[thm]{Proposition}

\theoremstyle{definition}
\newtheorem{defn}[thm]{Definition}

\theoremstyle{remark}
\newtheorem{note}[thm]{Remark}
\newtheorem{example}[thm]{Example}

\numberwithin{equation}{section}

\newcommand{\w}[1]{\widetilde{#1}}
\newcommand{\fra}[2]{\displaystyle{\frac{#1}{#2}}}
\newcommand{\R}{\mathbb{R}}

\begin{document}

\title[Bi-slant Submanifolds of Para Hermitian Manifolds]{Bi-slant Submanifolds of Para Hermitian Manifolds}

\author[P. Alegre]{Pablo Alegre}
 \address{Departamento de Econom\'{\i}a, M\'etodos Cuantitativos e Historia Econ\'omica, \'Area de Estad\'{\i}stica e Investigaci\'on Operativa.  Universidad Pablo de Olavide. Ctra. de Utrera km. 1, 41013 Sevilla, Spain}

\email[Corresponding author]{psalerue@upo.es}
\author[A. Carriazo]{Alfonso Carriazo}
 \address{Departamento de Geometr\'{i}a y Topolog\'{i}a. Universidad de Sevilla. c/ Tarfia s/n, 41012 Sevilla, Spain}

\email{carriazo@us.es}
\thanks{Both authors are partially supported by the MINECO-FEDER grant MTM2014-52197-P. They are members of the PAIDI group FQM-327 (Junta de Andaluc\'ia, Spain). The second author is also a member of the Instituto de Matem\'aticas de la Universidad de Sevilla (IMUS)}

\begin{abstract}
In this paper we introduce the notion of  bi-slant submanifolds of a para Hermitian manifold. They naturally englobe CR, semi-slant and hemi-slant submanifolds. We study their first properties and present a whole gallery of examples.
\end{abstract}



\maketitle

{\footnotesize{  2000 {\it Mathematics Subject Classification }:
 53C40, 53C50.
 }

\section{Introduction}
In \cite{slantchen}, B.-Y. Chen introduced slant submanifolds of an almost Hermitian manifold, as those submanifolds for which the angle $\theta$ between $JX$ and the tangent space is constant, for any tangent vector field $X$. They plays an intermediate role between complex submanifolds ($\theta=0$) and totally real ones ($\theta=\pi/2$). Since then, the study of slant submanifolds has produced an incredible amount of results and examples in two different ways: various ambient spaces and more general submanifolds.

On the one hand, J. L. Cabrerizo, A. Carriazo, L. M. Fern\'andez and M. Fern\'andez analyzed slant submanifolds of a Sasakian manifold in \cite{ccff}, and B. Sahin did in almost product manifolds in \cite{sahin}. The study of slant submanifolds in a semi-Riemannian manifold has been also initiated:
B.-Y. Chen, O. Garay and I. Mihai classified slant surfaces in Lorentzian complex space forms in \cite{Chen2} and \cite{Chen3}. K. Arslan, A. Carriazo, B.-Y. Chen and C. Murathan defined slant submanifolds of a neutral Kaehler manifold in \cite{accm}, while A. Carriazo and M. J. P\'erez-Garc\'{\i}a did in neutral almost contact pseudo-metric manifolds in \cite{cpg}. Moreover, M. A. Khan, K. Singh and V. A. Khan introduced slant submanifolds in LP-contact manifolds in \cite{khan}, and P. Alegre studied slant submanifolds of Lorentzian Sasakian and para Sasakian manifolds in \cite{Pablo1}. Finally, slant submanifolds of para Hermitian manifolds were defined in \cite{ac}.

On the other hand, some generalizations of both slant and CR submanifolds have also been defined in different ambient spaces, such as semi-slant \cite{papa} and \cite{ccff2}, hemi-slant \cite{sahin2}, bi-slant \cite{carriazo} or generic submanifolds \cite{ronsse}.

In this paper, we continue on this line, introducing semi-slant, hemi-slant and bi-slant submanifolds of para Hermitian manifolds.


\section{Preliminaries}

Let $\w M$ be a $2n$-dimensional semi-Riemannian manifold. If it is
endowed with a structure $(J,g)$, where $J$ is a
$(1,1)$ tensor, and $g$ is a semi-defined metric,
satisfying
\begin{equation}\label{cypc}
\begin{array}{ccc}
J^2 X= X, & g(J X,Y)+g(X,JY)=0,
\end{array}
\end{equation}
for any vector fields $X,Y$ on $\widetilde M$, it is called a {\it para Hermitian manifold}.
It is said to be {\it para Kaehler} if, in addition, $\widetilde \nabla J=0$, where $\widetilde\nabla$ is the Levi-Civita connection of $g$.

Let now $M$ be a submanifold of $(\w M,J,g)$. The Gauss and Weingarten formulas are given by
\begin{equation}\label{gauss}
\widetilde\nabla_XY=\nabla_XY+h(X,Y),
\end{equation}
\begin{equation}\label{Weingarten}
\widetilde\nabla_X V=-A_VX+\nabla^\perp_XN,
\end{equation}
for any tangent vector fields $X,Y$ and any normal vector field $V$, where $h$ is the second fundamental form of $M$, $A_V$ is the Weingarten endomorphism associated with $V$ and $\nabla^\perp$ is the normal connection.

And the Gauss and Codazzi equations are given by
\begin{equation}\label{gauss2}
\w R(X,Y,Z,W)=R(X,Y,Z,W)+g(h(X,Z),h(Y,W))-g(h(Y,Z),h(X,W)),
\end{equation}

\begin{equation}\label{codazzi}
(\w R(X,Y)Z)^\perp=(\w \nabla_X h)(Y,Z)-(\w\nabla_Yh)(X,Z),
\end{equation}
for any vectors fields $X,Y,Z,W$ tangent to $M$.

For every tangent vector field $X$, we write
\begin{equation}\label{decomp}
J X=PX+FX,
\end{equation}
where $PX$ is the tangential component of $JX$ and $FX$ is the normal one. And for every normal vector field $V$,
$$J V=tV+fV,$$
where $tV$ and $fV$ are the tangential and normal components of $JV$, respectively.

For such a submanifold of a para Kaehler manifold, taking the tangent and normal part and using the Gauss and Weingarten
formulas (\ref{gauss}) and (\ref{Weingarten})
\begin{equation}\label{nablap}
(\nabla_X P)Y=\nabla_XPY-P\nabla_XY=A_{FY}X+th(X,Y),\quad\quad
\end{equation}
\begin{equation}\label{nablan}
(\nabla_X F)Y=\nabla^\perp_XFY-F\nabla_XY=-h(X,PY)+fh(X,Y),
\end{equation}
for all tangent vector fields $X,Y$.



In \cite{ac}, we introduced the notion of slant submanifolds of para Hermitian manifolds, taking into account that we can not measure the angle for light-like vector fields:
\begin{defn}\cite{ac}
A submanifold $M$ of a para Hermitian manifold $(\w M,J,g)$ is called \emph{slant submanifold} if for every space-like or time-like tangent vector field $X$, the quotient $g(PX,PX)/g(JX,JX)$ is constant.
\end{defn}

\begin{note}
It is clear that, if $M$ is a para-complex submanifold, then $P\equiv J$, and so the above quotient is equal to $1$. On the other hand, if $M$ is totally real, then $P\equiv 0$ and the quotient equals $0$. Therefore, both para-complex and totally real submanifolds are particular cases of slant submanifolds. A neither para-complex nor totally real slant submanifold will be called \emph{proper slant}.
\end{note}

\

Three cases can be distinguished, corresponding to three different types of proper slant submanifolds:

\begin{defn}\cite{ac}
Let $M$ be a proper slant submanifold of a para Hermitian manifold $(\w M,J,g)$. We say that it is of
\begin{itemize}
\item[{\rm type 1}] if for any space-like (time-like) vector field $X$, $PX$ is time-like (space-like), and $\fra{|PX|}{|JX|}>1$,
\item[{\rm type 2}] if for any space-like (time-like) vector field $X$, $PX$ is time-like (space-like), and $\fra{|PX|}{|JX|}<1$,
\item[{\rm type 3}] if for any space-like (time-like) vector field $X$, $PX$ is space-like (time-like).
\end{itemize}
%
%
\end{defn}


These three types can be characterized as follows:
\begin{thm}\label{aannaa}\cite{ac}
Let $M$ be a submanifold of a para Hermitian manifold $(\w M, J,g)$. Then,
\begin{itemize}

\item[1)] $M$ is slant of type 1 if and only if for any space-like (time-like) vector field $X$, $PX$ is time-like (space-like), and there exists a constant $\lambda\in(1,+\infty)$ such that
\begin{equation}\label{anna2}
P^2=\lambda Id.
\end{equation}
We write $\lambda=\cosh^2\theta$, with $\theta >0$.
\item[2)] $M$ is slant of type 2 if and only if for any space-like (time-like) vector field $X$, $PX$ is time-like (space-like), and there exists a constant $\lambda\in(0,1)$ such that
\begin{equation}\label{anna3}
P^2=\lambda Id.
\end{equation}
We write $\lambda=\cos^2\theta$, with $0<\theta<2\pi$.
\item[3)] $M$ is slant of type 3 if and only if for any space-like (time-like) vector field $X$, $PX$ is space-like (time-like), and there exists a constant $\lambda\in(-\infty,0)$ such that
\begin{equation}\label{anna}
P^2=\lambda Id.
\end{equation}
We write $\lambda=-\sinh^2\theta$, with $\theta >0$.
\end{itemize}
In every case, we call $\theta$ the {\rm slant angle}.
\end{thm}

\begin{note}

It was proved in \cite{ac} that conditions (\ref{anna2}), (\ref{anna3}) and (\ref{anna}) also hold for every light-like vector field, as every light-like vector field can be decomposed as a sum of one space-like and one time-like vector field. Also, that every  slant submanifold of type 1 or 2 must be a neutral semi-Riemannian manifold.

\end{note}

Para-complex and totally real submanifolds can also be characterized by $P^2$. In \cite{ac} we did not consider that case, but it will be useful in the present study.

\begin{thm}
Let $M$ be a submanifold of a para Hermitian manifold $(\w M, J,g)$. Then,
\begin{itemize}
\item[1)] $M$ is a para-complex submanifold if and only if $P^2= Id.$
\item[2)] $M$ is a totally real submanifold if and only if $P^2=0.$
\end{itemize}
\end{thm}

\begin{proof}
If $M$ is para-complex, $P^2=J^2=Id$ directly. Conversely, if $P^2=Id$, from
$$g(JX,JX)=g(PX,PX)+g(FX,FX),$$
we have
$$-g(X,J^2X)=-g(X,P^2X)+g(FX,FX),$$
then $$-g(X,X)=-g(X,X)+g(FX,FX),$$
and hence $g(FX,FX)=0$, which implies $F=0$.

The second statement can be proved in a similar way.
\end{proof}


\section{Slant distributions}

In \cite{papa}, N. Papaghiuc introduced {\it slant distributions} in a Kaehler manifold. Given an almost Hermitian manifold, $(\w N, J,g)$, and a differentiable distribution $ D$, it is called a slant distribution if for any non zero vector $X\in D_x$, $x\in \w N$, the angle between $JX$ and the vector space $ D_x$ is constant, that is, is independent of the point $x$. If $P_DX$ is the projection of $JX$ over $ D$, they can be characterized as $P_ D^2=\lambda I$. This, together with the definition of slant submanifolds of a para Hermitian manifold, aims us to give the following:

\begin{defn}
A differentiable distribution $ D$ on a para Hermitian manifold $(\w M,J,g)$ is called a \emph{slant distribution} if for every non light-like $X\in D$, the quotient $g(P_DX,P_DX)/g(JX,JX)$ is constant .
\end{defn}

\

A distribution is called {\it invariant} if it is slant with slant angle $0$, that is if $g(P_DX,P_DX)/g(JX,JX)=1 $ for all non light-like $X\in D$. And it is called {\it anti-invariant} if $P_DX=0$ for all $X\in D$. In other case it is called {\it proper slant distribution}.


\

With this definition every one dimensional distribution defines an anti-invariant distribution in $\w M$, so we are just going to take under study non trivial slant distributions, that is with dimensions greater than $1$. Just like for slant submanifolds, we can consider three cases depending on the casual character of the implied vector fields.

Obviously, a submanifold $M$ is a slant submanifold if and only if $TM$ is a slant distribution.

\

\begin{defn}
Let $ D$ be a proper slant distribution of a para Hermitian manifold $(\tilde M,J,g)$. We say that it is of
\begin{itemize}
\item[{\rm type 1}] if for every space-like (time-like) vector field $X$, $P_DX$ is time-like (space-like), and $\fra{|P_DX|}{|JX|}>1$,
\item[{\rm type 2}] if for every space-like (time-like) vector field $X$,  if $P_ DX$ is time-like (space-like), and $\fra{|P_DX|}{|JX|}<1$,
\item[{\rm type 3}] if for every space-like (time-like) vector field $X$, $P_DX$ is space-like (time-like).
\end{itemize}
\end{defn}

\begin{thm}
Let $ D$ be a distribution of a para Hermitian metric manifold $\w M$. Then,
\begin{itemize}
\item[1)] $ D$ is a slant distribution of type 1 if and only for any space-like (time-like) vector field $X$, $P_DX$ is time-like (space-like), and there exits a constant $\lambda\in(1,+\infty)$ such that
\begin{equation}\label{aannaa2}
P_ D^2=\lambda I
\end{equation}
Moreover, in such a case, $\lambda=\cosh^2\theta$.
\item[2)] $ D$ is a slant distribution of type 2 if and only for any space-like (time-like) vector field $X$, $P_DX$ is time-like (space-like), and there exits a constant $\lambda\in(0,1)$ such that
\begin{equation}\label{aannaa3}
P_ D^2=\lambda I
\end{equation}
Moreover, in such a case, $\lambda=\cos^2\theta$.
\item[3)] $ D$ is a slant distribution of type 3 if and only for any space-like (time-like) vector field $X$, $P_DX$ is space-like (time-like), and there exits a constant $\lambda\in(0,+\infty)$ such that
\begin{equation}\label{aannaa4}
P_ D^2=\lambda I
\end{equation}
Moreover, in such a case, $\lambda=\sinh^2\theta$.
\end{itemize}
In each case, we call $\theta$ the {\rm slant angle}.
\end{thm}

\begin{proof}
In the first case, if $D$ is a slant distribution of type 1, for any space-like tangent vector field $X\in D$, $P_DX$ is time-like, and, by virtue of \eqref{cypc}, $JX$ also is. Moreover, they satisfy $|P_DX|/|JX|>1$. So, there exists $\theta>0$ such that
\begin{equation}\label{papana}
\cosh\theta=\fra{|P_DX|}{|JX|}=\fra{\sqrt{-g(P_DX,P_DX)}}{\sqrt{-g(JX,JX)}}.
\end{equation}

If we now consider $P_DX$, then, in a similar way, we obtain:
\begin{equation}\label{mericano}
\cosh\theta=\fra{|P_D^2X|}{|JP_DX|}=\fra{|P_D^2X|}{|P_DX|}.
\end{equation}

On the one hand,
\begin{equation}\label{papana1}
g(P_D^2X,X)=g(JP_DX,X)=-g(P_DX,JX)=-g(P_DX,P_DX)=|P_DX|^2.
\end{equation}
Therefore, using (\ref{papana}), (\ref{mericano}) and (\ref{papana1}) $$g(P_D^2X,X)=|P_DX|^2=|P_D^2X||JX|=|P_D^2 X||X|.$$

On the other hand, 
since both $X$ and $P_D^2X$ are space-like, it follows that they are collinear, that is $P_D^2X=\lambda X$. Finally, from (\ref{papana}) we deduce that $\lambda=\cosh^2\theta$.

Everything works in a similar way for any time-like tangent vector field $Y\in D$, but now, $P_DY$ and $JY$ are space-like and so, instead of (\ref{papana}) we should write:
$$\cosh\theta=\fra{|P_DY|}{|JY|}=\fra{\sqrt{g(P_DY,P_DY)}}{\sqrt{g(JY,JY)}}.$$

Since $P_D^2X=\lambda X$, for any space-like or time-like $X\in D$, it also holds for light-like vector fields and so we have that $P_D^2=\lambda Id_D$.

The converse is just a simple computation.

In the second case, if $D$ is a slant distribution of type 2, for any space-like or time-like vector field $X\in D$, $|P_DX|/|JX|<1$, and so there exists $\theta>0$ such that
\begin{equation*}\label{papana2}
\cos\theta=\fra{|P_DX|}{|JX|}=\fra{\sqrt{-g(P_DX,P_DX)}}{\sqrt{-g(JX,JX)}}.
\end{equation*}

Proceeding as before, we can prove that $g(P_D^2X, X)=| P_D^2X||X|$ and, as both $X$ and $P_D^2X$ are space-like vector fields, it follows that they are collinear, that is $P_D^2X=\lambda X$. Again, the converse is a direct computation.

Finally, if $D$ is a slant distribution of type 3, for any space-like vector field $X\in D$, $P_DX$ is also space-like, and there exists $\theta>0$ such that
\begin{equation*}\label{papana3}
\sinh\theta=\fra{|P_DX|}{|JX|}=\fra{\sqrt{g(P_DX,P_DX)}}{\sqrt{-g(JX,JX)}}.
\end{equation*}

Once more, we can prove that $g(P_D^2X, X)=| P_D^2X||X|$ and $P_D^2X=\lambda X$. And again, the converse is a direct computation.
\end{proof}

\

Remember that an {\it holomorphic distribution} satisfies $J D= D$, so every holomorphic distribution is a slant distribution with angle 0, but the converse is not true. And it is called {\it totally real distribution} if $J D\subseteq T^\perp M$, therefore every totally real distribution is anti-invariant but the converse does not always hold. For holomorphic and totally real distributions the following necessary conditions are easy to prove:

\begin{thm}
Let $ D$ be a distribution of a submanifold of a para Hermitian metric manifold $\w M$.
\begin{itemize}
\item[1)] If $ D$ is a holomorphic distribution then $|P_DX|=|JX|$, for all $X\in D$. 
\item[2)] If $ D$ is a totally real distribution then $|P_DX|=0$, for all $X\in D$.
\end{itemize}
\end{thm}

However the converse results do not hold if $D$ is not $TM$; in such a case $TM=D\oplus\nu$, and for a unit vector field $X$
$$JX=P_DX+P_\nu X+FX.$$
Therefore from
$$g(JX,JX)=g(P_DX, P_DX)+g(P_\nu X, P_\nu X)+g(FX,FX),$$
and $|P_DX|=|JX|$, in the case that $P_DX$ is also space-like, it is only deduced that
$$g(P_\nu X, P_\nu X)+g(FX,FX)=-2,$$
 or, in the case it is time-like, $$g(P_\nu X, P_\nu X)+g(FX,FX)=0.$$ So in general $FX\neq 0$, and $D$ is not invariant.

Similarly it can be shown that the converse of the second statement does not always hold.

\begin{thm}
Let $ M$ be a submanifold of a para Hermitian metric manifold $\w M$.
\begin{itemize}
\item[1)] The maximal holomorphic distribution is characterized as $D=\{X/ FX=0\}$.
\item[2)] The maximal totally real distribution is characterized as $D^\perp=\{X/ PX=0\}$.
\end{itemize}
\end{thm}

\begin{proof}
For the first statement, if a distribution $D$ is holomorphic, obviusly $F\rceil_D=0$. For the converse, consider $D=\{X/ FX=0\}$. We should prove that it is a holomorphic distribution. Let $X\in D$ be, $JX=TX$ is tangent to $M$, and
$$g(FJX,V)=g(J^2X,V)=g(X,V)=0,$$
for all $V\in T^\perp M$. Therefore $FJX=0$. That implies $JX\in D$ for all $X\in D$, so $D$ is holomorphic.

The second statement is trivial.

\end{proof}

\section{Bi-slant, semi-slant and hemi-slant submanifolds. Definition and examples.}
In  \cite{papa}, \emph{ semi-slant submanifolds} of an almost Hermitian manifold were introduced as those submanifolds  whose tangent space could be decomposed as a direct sum of two distributions, one totally real and the other a slant distribution. In \cite{carriazo}, \emph{anti-slant submanifolds} were introduced as those whose tangent space is decomposed as a direct sum of an anti-invariant and a slant distribution; they were called \emph{hemi-slant submanifolds} in \cite{sahin2}. Finally, in \cite{ccff2}, the authors defined \emph{ bi-slant submanifolds} with both distributions slant ones.

\begin{defn}
A submanifold $M$ of a para Hermitian manifold $(\w M,J,g)$ is called a \emph{ bi-slant submanifold} if the tangent space admits a decomposition $TM= D_1\oplus D_2$ with both $ D_1$ and $ D_2$ slant distributions.

It is called \emph{ semi-slant submanifold} if $TM= D_1\oplus D_2$ with $ D_1$ a holomorphic distribution and $ D_2$ a proper slant distribution. In such a case, we will write $D_1=D_T$.

And it is called \emph{ hemi-slant submanifold} if $TM= D_1\oplus D_2$ with $ D_1$ a totally real distribution and $ D_2$ a proper slant distribution. In such a case, we will write $D_1=D_\perp$.
\end{defn}

\begin{note}
As we have said before, being holomorphic (totally real) is a stronger condition than being slant with slant angle $0$ ($\pi/2$).

\end{note}

We write $\pi_i$ the projections over $D_i$ and $P_i=\pi_i\circ P$, $i=1,2$.

Let us consider two different para Kaehler structures over $\R^4$:
$$J=\left(\begin{array}{cccc}
0 & 1 & 0 & 0\\
1 & 0 & 0 & 0\\
0 & 0 & 0 & 1\\
0 & 0 & 1 & 0
\end{array}\right),\quad\quad g=\left(\begin{array}{cccc}
1 & 0 & 0 & 0\\
0& -1 & 0 & 0\\
0 & 0&1 & 0\\
0 & 0 & 0 &-1
\end{array}\right),$$
and
$$J_1=\left(\begin{array}{cccc}
0 & 0 & 1 & 0\\
0 & 0 & 0 & 1\\
1 & 0 & 0 & 0\\
0 & 1 & 0 & 0
\end{array}\right),\quad\quad g_1=\left(\begin{array}{cccc}
1 & 0 & 0 & 0\\
0& 1 & 0 & 0\\
0 & 0&-1 & 0\\
0 & 0 & 0 &-1
\end{array}\right).$$

Using the examples of slant submanifolds of $\R^4$ given in \cite{ac} and making products, we can obtain examples of bi-slant submanifolds in $\R^8$. To present different examples with all the combinations of slant distributions, we consider the following para Kaehler structures over $\R^8$:
$$J_2=\left(\begin{array}{cc}
J & \Theta\\
\Theta & J
\end{array}\right),\quad\quad g_2=\left(\begin{array}{cccc}
g & \Theta\\
\Theta & g
\end{array}\right),$$
$$J_3=\left(\begin{array}{cc}
J_1 & \Theta\\
\Theta & J
\end{array}\right),\quad\quad g_3=\left(\begin{array}{cccc}
g_1 & \Theta\\
\Theta & g
\end{array}\right),$$
$$J_4=\left(\begin{array}{cc}
J_1 & \Theta\\
\Theta & J_1
\end{array}\right),\quad\quad g_4=\left(\begin{array}{cccc}
g_1 & \Theta\\
\Theta & g_1
\end{array}\right),$$

\begin{example}\label{aw1}
For any $a,b,c,d \in \R$ with $a^2+b^2\neq 1$, and $c^2+d^2\neq 1$,
$$x(u_1,v_1,u_2,v_2)=(au_1, v_1,bu_1, u_1,cu_2, v_2,du_2, u_2)$$ defines a bi-slant submanifold in $(\R^8,J_2,g_2)$, with slant distributions $ D_1=\mbox{Span}\left\{\fra{\partial}{\partial u_1},\fra{\partial}{\partial v_1}\right\}$ and $ D_2=\mbox{Span}\left\{\fra{\partial}{\partial u_2},\fra{\partial}{\partial v_2}\right\}$. We can see the different types in the following table:

\begin{center}
\renewcommand
{\arraystretch}{1.5}
\begin{tabular}{|c | c|}
\hline
\begin{tabular}{c|c|c}
$\quad\quad\quad$&    $\quad\quad\quad\quad D_1\quad\quad\quad\quad$ & $\quad\quad\quad\quad D_2$
\end{tabular}
&
\\
\hline
\begin{tabular}{c|c|c}
type 1 & $a^2+b^2>1,\quad b^2<1$ & $c^2+d^2>1,\quad c^2<1$\\ \hline
type 2 & $a^2+b^2>1,\quad b^2>1$ & $c^2+d^2>1,\quad c^2>1$\\ \hline
time-like type 3 & $a^2+b^2<1$ & $c^2+d^2<1$
\end{tabular}
 &
$\begin{array}{c}
(\R^8,J_2,g_2)\\
P^2_1=\fra{a^2}{-1+a^2+b^2}Id_1\\

P^2_2=\fra{c^2}{-1+c^2+d^2}Id_2
\end{array}$
\\
\hline
\end{tabular}
\end{center}
\end{example}

\

\begin{note}
The decomposition of $TM$ in two slant distributions it is not unique, for example, if we choose $\tilde D_1=\mbox{Span}\left\{\fra{\partial}{\partial u_1},\fra{\partial}{\partial v_2}\right\}$ and $\tilde D_2=\mbox{Span}\left\{\fra{\partial}{\partial u_2},\fra{\partial}{\partial v_1}\right\}$ in the previous example, both distributions are anti-invariant, that is $P(\tilde D_1)=\tilde D_2$ and $P(\tilde D_2)=\tilde D_1$; therefore $P_1=0$ and $P_2=0$. However they are not totally real distributions.
\end{note}

\

\begin{example} Taking $a=0$ in the previous example we obtain a semi-slant submanifold, and taking $b=1$ we obtain a hemi-slant submanifold.
\end{example}

\

\begin{example}\label{comancheria}
For any $a,b,c,d$ with $a^2-b^2\neq 1$, $c^2-d^2\neq 1$
$$x(u_1,v_1,u_2,v_2)=(u_1, av_1,bv_1, v_1,u_2,cv_2,dv_2,v_2),$$ defines a bi-slant submanifold, with slant distributions $ D_1=\mbox{Span}\left\{\fra{\partial}{\partial u_1},\fra{\partial}{\partial v_1}\right\}$ and $ D_2=\mbox{Span}\left\{\fra{\partial}{\partial u_2},\fra{\partial}{\partial v_2}\right\}$. We can see the different types in the following table:

\

\begin{center}
\renewcommand
{\arraystretch}{1.5}
\begin{tabular}{|c | c|}
\hline
\begin{tabular}{c|c|c}
$\quad\quad\quad$&    $\quad\quad\quad\quad D_1\quad\quad\quad\quad$ & $\quad\quad\quad\quad D_2$
\end{tabular}
&
\\
\hline
\begin{tabular}{c|c|c}
type 1 & $b^2-a^2<1,\quad b^2>1$ & $d^2-c^2<1,\quad d^2>1$\\ \hline
type 2 & $b^2-a^2<1,\quad b^2<1$ & $d^2-c^2<1,\quad d^2<1$\\ \hline
space-like type 3 & $b^2-a^2>1$ & $d^2-c^2>1$
\end{tabular}
 &
$\begin{array}{c}
(\R^8,J_2,g_2)\\
P^2_1=\fra{a^2}{1+a^2-b^2}Id_1\\

P^2_2=\fra{c^2}{1+c^2-d^2}Id_2
\end{array}$
\\
\hline
\begin{tabular}{c|c|c}
type 1 & $b^2-a^2>1,\quad a^2>1$ & $d^2-c^2<1,\quad d^2>1$\\ \hline
type 2 & $b^2-a^2>1,\quad a^2<1$ & $d^2-c^2<1,\quad d^2<1$\\ \hline
space-like type 3 & $b^2-a^2<1$ & $d^2-c^2>1$
\end{tabular}
 &
$\begin{array}{c}
(\R^8,J_3,g_3)\\
P^2_1=\fra{a^2}{1+a^2-b^2}Id_1\\

P^2_2=\fra{c^2}{1+c^2-d^2}Id_2
\end{array}$
\\
\hline
\begin{tabular}{c|c|c}
type 1 & $b^2-a^2>1,\quad a^2>1$ & $d^2-c^2>1,\quad c^2>1$\\ \hline
type 2 & $b^2-a^2>1,\quad a^2<1$ & $d^2-c^2>1,\quad c^2<1$\\ \hline
space-like type 3 & $b^2-a^2<1$ & $d^2-c^2<1$
\end{tabular}
 &
$\begin{array}{c}
(\R^8,J_4,g_4)\\
P^2_1=\fra{a^2}{1+a^2-b^2}Id_1\\

P^2_2=\fra{c^2}{1+c^2-d^2}Id_2
\end{array}$
\\
\hline
\end{tabular}
\end{center}
\end{example}


\

Now we are interested in those bi-slant submanifolds of an almost para Hermitian manifold that are Lorentzian. Let us remember that the only odd dimensional slant distributions are the totally real ones, and that type 1 and 2 are neutral distributions. Taking this into account the only possible cases are the following:
\begin{itemize}
\item[i)] $M_1^{2s+1}$ with $TM=D_1\oplus D_2$, where $ D_1$ is a one dimensional, time-like, anti-invariant distribution and $ D_2$ is a space-like, type 3 slant distribution.
\item[ii)] $M_1^{2s+2}$ with $TM=D_1\oplus D_2$, where $ D_1$ is a two dimensional, neutral, slant distribution of type 1 or 2, and $ D_2$ is a space-like, type 3 slant distribution.
\end{itemize}

With examples \ref{aw1} and \ref{comancheria} we can obtain examples for the ii) case. It only remains to construct a case i) example.

\begin{example}\label{diecisiete}
Consider in $\R^6$ the almost para Hermitian structure given by $$J_5=\left(\begin{array}{cc}
J & \Theta\\
\Theta & \begin{array}{cc} 0 & 1 \\ 1 & 0 \end{array}
\end{array}\right),\quad\quad g_5=\left(\begin{array}{cccc}
g & \Theta\\
\Theta & \begin{array}{cc} 1 & 0 \\ 0 & -1 \end{array}
\end{array}\right).$$

For any $k>1$,
$$x(u,v,w)=(u, k\cosh v,v, k\sinh v,w,0)$$
defines a bi-slant submanifold in $(\R^6,J_5,g_5)$ with $ D_1=\mbox{Span}\left\{\fra{\partial}{\partial w}\right\}$ a totally real distribution and $D_2=\mbox{Span}\left\{\fra{\partial}{\partial u},
\fra{\partial}{\partial v}\right\}$ a type 3 slant distribution with $P_2^2=\fra{1}{k^2-1}Id\rceil_{ D_2}$.
\end{example}

\

We can present a bi-slant submanifold, with the same angle for both slant distributions, that is not a slant submanifold.

\begin{example}
The submanifold of $(\R^8,J_2,g_2)$ defined by
$$x(u_1,v_1,u_2,v_2)=(u_1, v_1+u_2,u_1, u_1,u_2,v_2,\sqrt{3}u_2,u_2-v_1),$$ is a bi-slant submanifold. The slant distributions are $ D_1=\mbox{Span}\left\{\fra{\partial}{\partial u_1},\fra{\partial}{\partial v_1}\right\}$ and $ D_2=\mbox{Span}\left\{\fra{\partial}{\partial u_2},\fra{\partial}{\partial v_2}\right\}$, with $P^2_1=\fra{1}{2}Id_1$ and $P^2_2=\fra{1}{2}Id_2$. It is not a slant submanifold.

\end{example}

\section{Semi-slant submanifolds of a para Kaehler manifold.}

It is always interesting to study the integrability of the involved distributions.

\begin{prop}\label{lugones}
Let $M$ be a semi-slant submanifolds of a para Hermitian manifold. Both the holomorphic and the slant distributions are $P$ invariant.
\end{prop}
\begin{proof}
Let be $TM=D_T\oplus D_2$ the decomposition with $D_1$ holomorphic and $D_2$ the slant distribution. Of course $D_T$ is invariant as $JD_T=D_T$ implies $PD_T=D_T$. Now, consider $X\in D_2$,
$$JX=P_1X+P_2X+FX.$$

Given $Y\in D_T$, $g(P_1X,Y)=g(JX,Y)=-g(X,JY)=0$, as $D_T$ is invariant. Moreover, for all $Z\in D_2$, $g(P_1X,Z)=0$. Therefore $P_1X=0$, and $PX=P_2X$, so $PD_2\subseteq D_2$.
\end{proof}

\

\begin{thm}\label{integrable1}
Let $M$ be a semi-slant submanifold of a para Kaehler manifold. The holomorphic distribution is integrable if and only if $h(X,JY)=h(JX,Y)$ for all $X,Y\in D_T$.
\end{thm}

\begin{proof}
For $X,Y\in D_T$, $PX=JX$, $FX=0$, $PY=JY$ and $FY=0$. From (\ref{nablan}) it follows $F[X,Y]=h(X,PY)-h(Y,PX)$. Then, $[X,Y]\in D_T$, that is $ D_T$ is integrable, if and only if $h(X,JY)=h(JX,Y)$.
\end{proof}

\

\begin{thm}\label{integrable2}
Let $M$ be a semi-slant submanifold of a para Kaehler manifold. The slant distribution is integrable if and only if
\begin{equation}\label{borges}
\pi_1(\nabla_X PY-\nabla_YPX)=\pi_1(A_{FY}X-A_{FX}Y),
\end{equation}
for all $X,Y\in D_2$, where $\pi_1$ is the projection over the invariant distribution $ D_T$.
\end{thm}

\begin{proof}
From (\ref{nablap}), $P_1\nabla_XY=\pi_1(\nabla_XPY-th(X,Y)-A_{FY}X)$. Then
$$P_1[X,Y]=\pi_1(\nabla_X PY-\nabla_YPX+A_{FX}Y-A_{FY}X).$$
Then (\ref{borges}) is equivalent to $P_1[X,Y]=0$. As $P_1[X,Y]=\pi_1 P[X,Y]=0$, it holds if and only if $P[X,Y]\in D_2$. Finally, from Theorem \ref{lugones} $D_2$ is $P$ invariant, so we obtain $[X,Y]\in D_2$.
\end{proof}



Now we study conditions for the involved distributions being totally geodesic.

\

\begin{prop}
Let $M$ be a semi-slant submanifold of a para Kaehler manifold $\w M$. If the holomorphic distribution $D_T$ is totally geodesic then $(\nabla_XP)Y=0$, and $\nabla_XY\in D_T$ for any $X,Y\in D_T$.
\end{prop}

\begin{proof}
For a para Kaehler manifold taking $X,Y\in D_T$, (\ref{nablap})-(\ref{nablan}) leads to
\begin{equation}
\nabla_XPY-P\nabla_XY-th(X,Y)=0,
\end{equation}
\begin{equation}
-F\nabla_XY+h(X,PY)-fh(X,Y)=0.
\end{equation}
If $D_T$ is totally geodesic, $(\nabla_XP)Y=0$ and $F\nabla_XY=0$, which imply the result.
\end{proof}

\

Note that for semi-slant submanifolds of para Kaehler manifolds, on the opposite that for Kaehler manifolds \cite{papa}.

\begin{prop}\label{lemasemi}
Let $M$ be a semi-slant submanifold of a para Kaehler manifold $\w M$. The slant distribution $D_2$ is totally geodesic if and only if $(\nabla_XF)Y=0$, and $(\nabla_XP)Y=A_{FY}X$ for any $X,Y\in D_2$.
\end{prop}

\begin{proof}
If $D_2$ is a totally geodesic distribution, from (\ref{nablap}) and (\ref{nablan}), taking $X,Y\in D_2$
\begin{equation}\label{merry2}
\nabla_XPY-A_{FY}X-P\nabla_XY=0,
\end{equation}
\begin{equation}\label{navidad2}
\nabla_X^\perp FY-F\nabla_XY=0.
\end{equation}
which implies the given conditions.
On the converse, if $(\nabla_XP)Y=A_{FY}X$, then $th(X,Y)=0$, which implies $Jh(X,Y)=fh(X,Y)$. From (\ref{nablan}) and $\nabla F=0$, it holds $h(X,PY)=nh(X,Y)$. Then for $PY\in D_2$
$$\lambda h(X,Y)=h(X,P^2Y)^=f^2h(X,Y)=J^2h(X,Y)=h(X,Y),$$
and as $D_2$ is a proper slant distribution, $\lambda\neq 1$, it must be $h(X,Y)=0$ for all $X,Y\in D_2$.
\end{proof}

\

Given two orthogonal distributions $ D_1$ and $ D_2$ over a submanifold, it is called $ D_1- D_2$-{\it mixed totally geodesic} if $h(X,Y)=0$ for all $X\in D_1$ , $Y\in D_2$.

\begin{prop}
Let $M$ be a semi-slant submanifold of a para Hermitian manifold $\w M$. $M$ is mixed totally geodesic if and only if $A_NX\in D_i$ for any $X\in D_i$, $N\in T^\perp M$, $i=1,2$.
\end{prop}

\begin{proof}
If $M$ is $ D_T- D_2$ mixed totally geodesic, for any $X\in D_T$, $Y\in D_2$,
$$g(A_NX,Y)=g(h(X,Y),N)=0,$$
which implies $A_NX\in D_T.$ The same proof is valid for $X\in D_2$ and for the converse.
\end{proof}

\begin{prop}
Let $M$ be a semi-slant submanifold of a para Kaehler manifold $\w M$. If $\nabla F=0$, then either $M$ is $ D_T- D_2$-mixed totally geodesic or $h(X,Y)$ is a eigenvector of $f^2$ associated with the eigenvalue 1, for all $X\in D_T$, $Y\in D_2$.
\end{prop}

\begin{proof}
Let be $X\in D_T$, $Y\in D_2$, if $\nabla F=0$, from (\ref{nablan}), $fh(X,Y)=h(X,PY)$.

As $ D_T$ is holomorphic, that is $J$-invariant, $ D_2$ is $P$-invariant. Therefore,
$$f^2h(X,Y)=fh(X,PY)=h(X,P^2 Y)=h(X,P^2_2 Y)=\lambda h(X,Y),$$
with $\lambda=\cosh^2\theta$ ($\cos^2\theta, \sinh^2\theta$ respectively). But also
$$f^2h(Y,X)=fh(Y,PX)=h(Y,P^2X)=h(Y,X).$$
From both equations, either $h(X,Y)=0$ or it is a eigenvalue of $f^2$ associated with $\lambda=1$.
\end{proof}

\begin{prop}
Let $M$ be a mixed totally geodesic semi-slant submanifold of a para Kaehler manifold $\w M$. If $ D_T$ is integrable, then $PA_NX=A_NPX$, for all $X\in D_T$ and $N\in T^\perp M$.
\end{prop}

\begin{proof}
From Theorem \ref{integrable1}, $h(X,JY)=h(Y,JX)$ for all $X,Y\in D_T$,
$$g(JA_NX,Y)=-g(A_NX,PY)=-g(N,h(X,PY))=-g(N,h(Y,PX))=-g(A_NPY,Y).$$
And given $Z\in D_2$,
$$g(JA_NX,Z)=-g(A_NX,PZ)=-g(N,h(X,PZ))=0,$$ because $M$ is mixed totally geodesic. From both equations $PA_NX=A_NPX$ what finishes the proof.
\end{proof}

\

Finally the mixed-totally geodesic characterization can be summarized with

\begin{thm}\label{mixed}
Let $M$ be a proper semi-slant submanifold of a para Kaehler manifold $\w M$. $M$ is $D_T- D_2$-mixed totally geodesic if and only if $(\nabla_XP)Y=A_{FY}X$ and $(\nabla_XF)Y=0$, for all $X,Y$ in different distributions.
\end{thm}

\begin{proof}
On the one hand, if $M$ is $D_T- D_2$-mixed totally geodesic, let be $X,Y$ belonging to different distributions. From (\ref{nablap}) and (\ref{nablan}), both conditions are deduced.

On the other hand, from (\ref{nablap}) and $(\nabla_XP)Y=A_{FY}X$, it is deduced $th(X,Y)=0$. And from (\ref{nablan}) and $(\nabla_XF)Y=0$ it is deduced
\begin{equation}
h(X,PY)=fh(X,Y),
\end{equation}
for all $X,Y$ in different distributions.

Therefore, for $X\in D_T$ and $Y\in D_2$
$$f^2h(X,Y)=h(X,P^2Y)=\lambda h(X,Y)$$
and also
$$f^2h(Y,X)=h(Y,P^2X)=h(Y,X).$$
As $M$ is a proper semi-slant submanifold, $\lambda\neq 1$, and $h(X,Y)=0$ so $M$ is mixed totally geodesic.
\end{proof}
\section{Hemi-slant submanifolds of a para Kaehler manifold.}

We will also study the integrability of the involved distributions for a hemi-slant submanifold.

\begin{prop}
Let $M$ be a hemi-slant submanifolds of a para Hermitian manifold. The slant distribution is $P$ invariant.
\end{prop}
\begin{proof}
Let be $TM=D_\perp\oplus D_2$ the decomposition with $D_\perp$ totally real and $D_2$ the slant distribution. Consider $X\in D_2$,
$$JX=P_1X+P_2X+FX.$$

Given $Y\in D_\perp$, $g(PX,Y)=g(JX,Y)=-g(X,JY)=0$, as $D_\perp$ is totally real, therefore $PD_2\subseteq D_2$. As $P_2^2=\lambda Id$, given $X\in D_2$, $X=P\left(\frac{1}{\lambda}X\right)$, then $X\in PD_2$ and it is proved that $PD_2=D_2$.
\end{proof}

\

\begin{lem}
Let $M$ be a hemi-slant submanifold of a para Kaehler manifold. The totally real distribution is integrable if and only if $A_{FX}Y=A_{FY}X$ for all $X,Y\in D_\perp$.
\end{lem}

\begin{proof}
For $X,Y\in D_\perp$, $PX=0$, $JX=FX$, $PY=0$ and $JY=FY$. From (\ref{nablap}) it follows $P[X,Y]=A_{FX}Y-A_{FY}X$. Then $[X,Y]\in D_\perp$, that is $ D_\perp$ is integrable, if and only if $A_{FX}Y=A_{FY}X$.
\end{proof}

\

The following result was known for hemi-slant submanifolds of Kaehler manifolds, \cite{sahin2}. We obtain the equivalent one for hemi-slant submanifolds of para Kaehler manifolds:

\begin{thm}
Let $M$ be a hemi-slant submanifold of a para Kaehler manifold. The totally real distribution is always integrable.
\end{thm}

\begin{proof}
From the previous lemma it is enough to prove $g(A_{FX}Y,Z)=g(A_{FY}X,Z)$, for $X,Y\in D_\perp$ and $Z$ tangent.
Then, $$g(A_{FX}Y,Z)=g(h(Y,Z),FX)=g(-th(Y,Z),X)=$$
using (\ref{nablap})
$$=g(P\nabla_ZY+A_{FY}Z,X)=g(A_{FY}Z,X)=g(A_{FY}X,Z),$$
which finishes the proof.
\end{proof}

\

Now we study the integrability of the slant distribution.

\begin{thm}
Let $M$ be a hemi-slant submanifold of a para Kaehler manifold. The slant distribution is integrable if and only if
\begin{equation}\label{borges2}
\pi_1(\nabla_X PY-\nabla_YPX)=\pi_1(A_{FY}X-A_{FX}Y),
\end{equation}
for all $X,Y\in D_2$, where $\pi_1$ is the projection over the totally real distribution $ D_\perp$.
\end{thm}
The proof is analogous to the one of Theorem \ref{integrable2}.


\

\begin{lem}
Let $M$ be a hemi-slant submanifold of a para Kaehler manifold $\w M$. The totally real distribution $D_\perp$ is totally geodesic if and only if $(\nabla_XF)Y=0$, and $P\nabla_XY=-A_{FY}X$ for any $X,Y\in D_\perp$.
\end{lem}

\begin{proof}
From (\ref{nablap}) and (\ref{nablan}) for $X,Y\in D_\perp$
\begin{equation}
-P\nabla_XY-A_{FY}X-th(X,Y)=0,
\end{equation}
\begin{equation}
\nabla_X^\perp FY-F\nabla_XY-fh(X,Y)=0,
\end{equation}
which imply the given conditions.
\end{proof}

The same proof of Lemma \ref{lemasemi} is valid for the slant distribution of a hemi-slant distribution.

\begin{lem}
Let $M$ be a hemi-slant submanifold of a para Kaehler manifold $\w M$. The slant distribution $D_2$ is totally geodesic if and only if $(\nabla_XF)Y=0$, and $P\nabla_XY=-A_{FY}X$ for any $X,Y\in D_2$.
\end{lem}

\

Remember that the classical De Rham–Wu Theorem, \cite{wu} \cite{sebastian}, says that
two orthogonally, complementary and geodesic foliations (called a direct product
structure) in a complete and simply connected semi-Riemannian manifold give
rise to a global decomposition as a direct product of two leaves. Therefore, from the previous lemmas it is directly deduced:

\begin{thm}
Let $M$ be a complete and simply connected hemi-slant submanifold of a para Kaehler manifold $\w M$. Then, $M$ is locally the product of the integral submanifolds of the slant distributions if and only if $(\nabla_XF)Y=0$, and $P\nabla_XY=-A_{FY}X$ for both any $X,Y\in D_\perp$ or $X,Y\in D_2$.
\end{thm}

\

Finally, we can also study when a hemi-slant submanifold is mixed totally geodesic. We get a result similar to Proposition \ref{mixed}, but now the proof is much more easier.

\begin{prop}
Let $M$ be a hemi-slant submanifold of a para Kaehler manifold $\w M$. $M$ is $D_\perp - D_2$-mixed totally geodesic if and only if $(\nabla_XP)Y=A_{FY}X$ and $(\nabla_XF)Y=0$, for all $X,Y$ in different distributions.
\end{prop}

\begin{proof}
Again, if $M$ is $D_\perp - D_2$-mixed totally geodesic, and $X,Y$ belong to different distributions, from (\ref{nablap}) and (\ref{nablan}), both conditions are deduced.

Now, if we suppose both conditions, from (\ref{nablap}) and (\ref{nablan}), it is deduced $th(x,Y)=0$ and $h(X,PY)=fh(X,Y)$. So, taking $X\in D_2$ and $Y\in D_\perp$, we get $th(X,Y)=0$ and $fh(X,Y)=0$. Therefore $h(X,Y)=0$ and $M$ is mixed totally geodesic.
\end{proof}

\section{CR-submanifolds of a para Kaehler manifold.}
CR-submanifolds have been intensively studied in many environments. Moreover, there are also some works about CR submanifolds of para Kaehler manifolds, \cite{adelarosca}. A submanifold $M$ of an almost para Hermitian manifold is called a {\it CR-submanifold} if the tangent bundle admits a decomposition $TM= D\oplus D^\perp$ with $ D$ an holomorphic distribution, that is $J D= D$, and $ D^\perp$ a totally real one, that is $J D\subseteq T^\perp M$.

Now we make a study similar to the one made for generalized complex space forms in \cite{barros}.

Examples of CR-submanifolds can be obtained from Example \ref{aw1}. Taking $a=1, d=0$, $D_1=\mbox{Span}\left\{\fra{\partial}{\partial u_1},\fra{\partial}{\partial v_1}\right\}$ is a totally real distribution and $ D_2=\mbox{Span}\left\{\fra{\partial}{\partial u_2},\fra{\partial}{\partial v_2}\right\}$ is an holomorphic distribution. Moreover:
\begin{itemize}
\item[1)] $D_1$ is type 1 if $b^2<1$
\item[2)] $D_1$ is type 2 if $b^2>1$,
\item[3)] $D_2$ is type 2 if $c^2>1$,
\item[4)] $D_2$ is type 3 if $c^2<1$.
\end{itemize}

So we got examples of CR-submanifolds of type 1-2, 1-3, 2-2 and 2-3. Taking $a=0, d=1$ we can obtain 2-1, 2-2, 3-1 and again 3-2 examples.

\



\

For a para Kaehler manifold with constant holomorphic curvature for every non-light-like vector field, that is $\w R(X,JX,JX,X)=c$, the curvature tensor is given by
\begin{equation}\label{mejicANA}
\w R(X,Y)Z=\fra{c}{4}\{g(X,Z)Y-g(Y,Z)X+g(X,JZ)JY-g(Y,JZ)JX+2g(X,JY)JZ\};
\end{equation}
such a manifold is called a {\it para complex space form}.

\begin{thm}
Let $M$ be a slant submanifold of a para Kaehler space form $\w M(c)$. Then, $M$ is a proper CR submanifold if and only if the maximal holomorphic subspace $ D_p=T_pM\bigcap JT_pM$, $p\in M$, defines a non trivial differentiable distribution on $M$ such as $$\w R( D, D, D^\perp, D^\perp)=0,$$
where $ D^\perp$ denotes the orthogonal complementary of $ D$ on $TM$.
\end{thm}

\begin{proof}
If $M$ is a CR submanifold, from (\ref{mejicANA}) $$\w R(X,Y)Z=2g(X,JY)JZ,$$ for all $X,Y\in  D$ and $Z\in  D^\perp$, and this is normal to $M$; therefore the equality holds.

On the other hand, let $ D_p=T_pM\bigcap JT_pM$ be and suppose $\w R( D, D, D^\perp, D^\perp)=0.$ Again from (\ref{mejicANA}),
$$\w R(X,JX,Z,W)=\fra{c}{2}g(X,X)g(JZ,W),$$
for every $X\in D$, $Z,W\in D^\perp$. Taking $X\neq 0$ a non-light-like vector, it follows that $g(JZ,W)=0$. Then $JZ$ is orthonormal to $ D^\perp$ and it is normal. Therefore $ D^\perp$ is anti-invariant and $M$ is a CR submanifold.
\end{proof}

\

There is  a well known result for CR submanifolds of a complex space form $\w M(c)$ \cite{barros} establishing that if the invariant distribution is integrable, then the holomorphic sectional curvature determined by a unit vector field, $X\in D$, is upper bounded by the global holomorphic sectional curvature. That is, for every unit vector field $X$
$$H(X)=R(X,JX,JX,X)\leq c.$$
The situation in the semi Riemannian case, for a para complex space form is completely different. From (\ref{mejicANA}) and (\ref{gauss2}), for every non-light-like tangent unit vector field $X$ it holds
$$R(X,JX,JX,X)=c+g(h(X,X),h(JX,JX))-g(h(X,JX),h(X,JX)).$$
Now, if $ D$ is integrable, from Theorem \ref{integrable1}, $h(JX,JX)=h(X,J^2 X)=h(X,X)$, and then
$$H(X)=c+\|h(X,X)\|^2-\|h(X,JX)\|^2.$$

\

A submanifold is called {\it totally umbilical} if there exists a normal vector field $L$ such as $h(X,Y)=g(X,Y)L$ for all tangent vector fields $X,Y$. {\it Totally geodesic} submanifolds are particular cases with $L=0$.

\begin{thm}
There not exits proper CR totally umbilical submanifolds of a para complex space form $\w M(c)$ with $c\neq 0$.
\end{thm}

\begin{proof}
From (\ref{mejicANA}) it follows
$$ (\w R(X,Y)Z)^\perp=\fra{c}{4}\{g(X,JZ)FY-g(Y,JZ)FX+2g(X,JY)FZ\},$$ for all $X,Y,Z$ tangent vectors fields. Supposing $M$ is a proper CR submanifold we can choose two non-light-like vector fields $X\in D$ and $Z\in D^\perp$; for them $$ (\w R(X,JX)Z)^\perp=\fra{c}{2}g(X,X)FZ.$$
But for a totally umbilical submanifold, Codazzi's equation (\ref{codazzi}) gives
$$(\w R(X,Y)Z)^\perp=\nabla_X^\perp g(Y,Z)L-g(\nabla_XY,Z)L-g(Y,\nabla_X Z)L-\nabla_Y^\perp g(X,Z)L+g(\nabla_YX,Z)L+g(X,\nabla_YZ)L=0.$$ Comparing both equations, if $c\neq0$, it follows $FZ=0$ which is a contradiction.
\end{proof}

\

Moreover the same proof is valid for asserting:

\begin{cor}
There not exits proper semi slant totally umbilical submanifolds of a para complex space form $\w M(c)$ with $c\neq 0$.
\end{cor}

\end{document}